\documentclass{article}
\usepackage{a4wide}
\usepackage{amsthm}
\usepackage{graphicx}
\usepackage{amssymb}
\usepackage{amsmath}
\usepackage{ascmac}
\usepackage{setspace}
\usepackage{float}
\usepackage[dvips,usenames]{color}
\usepackage{colortbl}
\usepackage{algorithm}
\usepackage{algorithmic}
\usepackage{setspace}
\newtheorem{Theorem}[equation]{Theorem}
\newtheorem{Corollary}[equation]{Corollary}

\theoremstyle{definition}
\newtheorem{Definition}[equation]{Definition}

\theoremstyle{remark}
\newtheorem{Remark}[equation]{Remark}
\numberwithin{equation}{section}
\newtheorem{Claim}[equation]{Claim}

\DeclareMathOperator{\ev}{ev}

\DeclareMathOperator{\ad}{ad}

\newcommand{\ve}{\varepsilon}

\allowdisplaybreaks

\begin{document}
\title{The surjectivity of the evaluation map of the affine super Yangian}
\author{Mamoru Ueda}
\date{}
\maketitle
\begin{abstract}
There exists a homomorphism from the affine super Yangian to the completion of the universal enveloping algebra of $\widehat{\mathfrak{gl}}(m|n)$, called the evaluation map. In this paper, we show that the image of this homomorphism is dense. Via this homomorphism, we obtain irreducible representations of the affine super Yangian. 
\end{abstract}
\section{Introduction}
Drinfeld (\cite{D1}, \cite{D2}) defined the Yangian of a finite dimensional simple Lie algebra $\mathfrak{g}$ in order to obtain a solution of the Yang-Baxter equation. The Yangian is a quantum group which is a deformation of the current algebra $\mathfrak{g}[z]$. The definition of Yangian naturally extends to the case that $\mathfrak{g}$ is a Kac-Moody Lie algebra. In the case when $\mathfrak{g}$ is $\widehat{\mathfrak{sl}}(n)$, the Yangian associated with $\mathfrak{g}$ is a deformation of the universal central extension of $\mathfrak{sl}(n)[u^{\pm 1}, v]$ (\cite{Gu1}). 

It is well-known that the Yangians are closely related to $W$-algebras. At first, Ragoucy and Sorba (\cite{RS}) showed that there exist surjective homomorphisms from Yangians of type $A$ to rectangular finite $W$-algebras of type $A$. More generally, Brundan and Kleshchev (\cite{BK}) constructed a surjective homomorphism from a shifted Yangian, a subalgebra of the Yangian of type $A$, to finite $W$-algebras of type $A$. In the affine case, using a geometric realization of the Yangian, Schiffman and Vasserot (\cite{SV}) have constructed a surjective homomorphism from the Yangian of $\widehat{\mathfrak{gl}}(1)$ to the universal enveloping algebra of the principal $W$-algebra of type $A$, and proved the celebrated AGT conjecture (\cite{Ga}, \cite{BFFR}).

The relationship between Yangians and $W$-algebras are also studied in the case of Lie superalgebras. Provided that $\mathfrak{g}$ is $\mathfrak{sl}(m|n)$, Stukopin defined the Yangian of $\mathfrak{sl}(m|n)$, called the super Yangian (see \cite{S} and \cite{G}). It is a deformation of the current algebra $\mathfrak{sl}(m|n)[z]$. In the affine super setting, the affine super Yangian was defined in \cite{U2} and is a deformation of $\widehat{\mathfrak{sl}}(m|n)[z]$. 

In the finite super case, Briot and Ragoucy \cite{BR} constructed a surjective homomorphism from the super Yangian to rectangular finite $W$-superalgebras of type $A$. In the recent paper \cite{GLPZ}, Gaberdiel, Li, Peng and H. Zhang defined the Yangian $\widehat{\mathfrak{gl}}(1|1)$ for the affine Lie superalgebra $\widehat{\mathfrak{gl}}(1|1)$ and obtained the similar result to that of \cite{SV} in the super setting. Moreover, \cite{U3} gives a surjective homomorphism from the affine super Yangian to the universal enveloping algebra of the rectangular $W$-superalgebra of type $A$. Thus, rational representations of rectangular $W$-superalgebras of type $A$ can be seen as those of the affine super Yangians.

However, we know only a little about irreducible representations of the affine super Yangian. In the case when $\mathfrak{g}$ is $\widehat{\mathfrak{sl}}(n)$, the easiest irreducible representations of the affine Yangian are obtained by the pullback of irreducible highest weight representations of $\widehat{\mathfrak{gl}}(n)$ since there exists a surjective homomorphism from the affine Yangian to the completion of the universal enveloping algebra of $\widehat{\mathfrak{gl}}(n)$ (\cite{Gu1}, \cite{K1}, and \cite{K2}). In \cite{K2}, Kodera showed that the image of this homomorphism topologically generates the completed universal enveloping algebra by using a braid group action on the affine Yangian. It is natural to try to obtain irreducible representations of the affine super Yangian in the similar way. In \cite{U2}, we have constructed a homomorphism from the affine super Yangian to the completion of the universal enveloping algebra of $\widehat{\mathfrak{gl}}(m|n)$. However, we cannot prove that he image of this homomorphism is dense in the similar way to the one in \cite{K2} since we have no braid group actions on the affine super Yangian. In this paper, we show the statement in the more primitive way. Owing to this result, we obtain irreducible representations of the affine super Yangian via this homomorphism.
\section*{Acknowledgement}
The author wishes to express his gratitude to his supervisor Tomoyuki Arakawa for suggesting lots of advice to improve this paper. The author is also grateful for the support and encouragement of Ryosuke Kodera. I am also grateful for Oleksandr Tsymbaliuk for pointing out some mistakes. This work was supported by Iwadare Scholarship. 
\section{Affine Super Yangians}
First, we recall the definition of the affine super Yangian (see \cite{U2}). In this paper, we denote $xy+yx$ as $\{x,y\}$. Moreover, we set 
\begin{gather*}
p(i)=\begin{cases}
0&(1\leq i\leq m),\\
1&(m+1\leq i\leq m+n).
\end{cases}
\end{gather*}
\begin{Definition}\label{Def}
Suppose that $m, n\geq2$ and $m\neq n$. The affine super Yangian $Y_{\ve_1,\ve_2}(\widehat{\mathfrak{sl}}(m|n))$ is the associative superalgebra over $\mathbb{C}$ generated by $x_{i,r}^{+}, x_{i,r}^{-}, h_{i,r}$ $(i \in\{0,1,\cdots,m+n-1\}, r \in \mathbb{Z}_{\geq 0})$ with parameters $\ve_1, \ve_2 \in \mathbb{C}$ subject to the relations:
\begin{gather}
	[h_{i,r}, h_{j,s}] = 0, \label{eq1.1}\\
	[x_{i,r}^{+}, x_{j,s}^{-}] = \delta_{ij} h_{i, r+s}, \label{eq1.2}\\
	[h_{i,0}, x_{j,r}^{\pm}] = \pm a_{ij} x_{j,r}^{\pm},\label{eq1.3}\\
	[h_{i, r+1}, x_{j, s}^{\pm}] - [h_{i, r}, x_{j, s+1}^{\pm}] 
	= \pm a_{ij} \dfrac{\varepsilon_1 + \varepsilon_2}{2} \{h_{i, r}, x_{j, s}^{\pm}\} 
	- m_{ij} \dfrac{\varepsilon_1 - \varepsilon_2}{2} [h_{i, r}, x_{j, s}^{\pm}],\label{eq1.4}\\
	[x_{i, r+1}^{\pm}, x_{j, s}^{\pm}] - [x_{i, r}^{\pm}, x_{j, s+1}^{\pm}] 
	= \pm a_{ij}\dfrac{\varepsilon_1 + \varepsilon_2}{2} \{x_{i, r}^{\pm}, x_{j, s}^{\pm}\} 
	- m_{ij} \dfrac{\varepsilon_1 - \varepsilon_2}{2} [x_{i, r}^{\pm}, x_{j, s}^{\pm}],\label{eq1.5}\\
	\sum_{w \in \mathfrak{S}_{1 + |a_{ij}|}}[x_{i,r_{w(1)}}^{\pm}, [x_{i,r_{w(2)}}^{\pm}, \dots, [x_{i,r_{w(1 + |a_{ij}|)}}^{\pm}, x_{j,s}^{\pm}]\dots]] = 0\  (i \neq j),\label{eq1.6}\\
	[x^\pm_{i,r},x^\pm_{i,s}]=0\ (i=0, m),\label{eq1.7}\\
	[[x^\pm_{i-1,r},x^\pm_{i,0}],[x^\pm_{i,0},x^\pm_{i+1,s}]]=0\ (i=0, m),\label{eq1.8}
\end{gather}
where\begin{gather*}
a_{ij} =
	\begin{cases}
	{(-1)}^{p(i)}+{(-1)}^{p(i+1)}  &\text{if } i=j, \\
	         -{(-1)}^{p(i+1)}&\text{if }j=i+1,\\
	         -{(-1)}^{p(i)}&\text{if }j=i-1,\\
	        1 &\text{if }(i,j)=(0,m+n-1),(m+n-1,0),\\
		0  &\text{otherwise,}
	\end{cases}\\
	 m_{i,j}=
	\begin{cases}
	-{(-1)}^{p(i+1)} &\text{if } i=j + 1,\\
		{(-1)}^{p(i)} &\text{if } i=j - 1,\\
	        -1 &\text{if }(i,j)=(0,m+n-1),\\
	        1 &\text{if }(i,j)=(m+n-1,0),\\
		0  &\text{otherwise,}
	\end{cases}
\end{gather*}
and the generators $x^\pm_{m, r}$ and $x^\pm_{0, r}$ are odd and all other generators are even.
\end{Definition}
One of the difficulty of Definition~\ref{Def} is that the number of generators is infinite. There exists a presentation of the affine super Yangian such that the number of generators are finite. 

First, we show that the affine super Yangian is generated by $h_{i,r}$ and $x^\pm_{i,r}$\ $(i \in \{0,1,\cdots,m+n-1\}, r=0,1)$. Let us set $\tilde{h}_{i,1} = {h}_{i,1} - \dfrac{\ve_1 + \ve_2}{2} h_{i,0}^2$. Using $\tilde{h}_{i,1}$, we can rewrite \eqref{eq1.4} as
\begin{equation}
[\tilde{h}_{i,1}, x_{j,r}^{\pm}] = \pm a_{ij}\left(x_{j,r+1}^{\pm}-m_{ij}\dfrac{\varepsilon_1 - \varepsilon_2}{2} x_{j, r}^{\pm}\right).\label{11111}
\end{equation}
By \eqref{11111}, we find that $Y_{\ve_1,\ve_2}(\widehat{\mathfrak{sl}}(m|n))$ is generated by $x_{i,r}^{+}, x_{i,r}^{-}, h_{i,r}$ $(i \in \{0,1,\cdots,m+n-1\}, r = 0,1)$. In fact, by \eqref{11111} and \eqref{eq1.2}, we have the following relations;
\begin{gather}
x^\pm_{i,r+1}=\pm\dfrac{1}{a_{i,i}}[\tilde{h}_{i,1},x^\pm_{i,r}],\qquad h_{i,r+1}=[x^+_{i,r+1},x^-_{i,0}]\quad\text{if}\quad i\neq m,0,\label{eq1297}\\
x^\pm_{i,r+1}=\pm\dfrac{1}{a_{i+1,i}}[\tilde{h}_{i+1,1},x^\pm_{i,r}]+m_{i+1,i}\dfrac{\varepsilon_1 - \varepsilon_2}{2} x_{i, r}^{\pm},\qquad h_{i,r+1}=[x^+_{i,r+1},x^-_{i,0}]\quad\text{if}\quad i=m,0\label{eq1298}
\end{gather}
for all $r\geq1$. In the following theorem, we construct the minimalistic presentation of the affine super Yangian $Y_{\ve_1,\ve_2}(\widehat{\mathfrak{sl}}(m|n))$ whose generators are $x_{i,r}^{+}, x_{i,r}^{-}, h_{i,r}$ $(i \in \{0,1,\cdots,m+n-1\}, r = 0,1)$. 
\begin{Theorem}[\cite{U2}, Theorem~3.17]\label{Mini}
Suppose that $m, n\geq2$ and $m\neq n$. The affine super Yangian $Y_{\ve_1,\ve_2}(\widehat{\mathfrak{sl}}(m|n))$ is isomorphic to the superalgebra generated by $x_{i,r}^{+}, x_{i,r}^{-}, h_{i,r}$ $(i \in\{0,1,\cdots,m+n-1\}, r = 0,1)$ subject to the defining relations:
\begin{gather}
[h_{i,r}, h_{j,s}] = 0,\label{eq2.1}\\
[x_{i,0}^{+}, x_{j,0}^{-}] = \delta_{ij} h_{i, 0},\label{eq2.2}\\
[x_{i,1}^{+}, x_{j,0}^{-}] = \delta_{ij} h_{i, 1} = [x_{i,0}^{+}, x_{j,1}^{-}],\label{eq2.3}\\
[h_{i,0}, x_{j,r}^{\pm}] = \pm a_{ij} x_{j,r}^{\pm},\label{eq2.4}\\
[\tilde{h}_{i,1}, x_{j,0}^{\pm}] = \pm a_{ij}\left(x_{j,1}^{\pm}-m_{ij}\dfrac{\varepsilon_1 - \varepsilon_2}{2} x_{j, 0}^{\pm}\right),\label{eq2.5}\\
[x_{i, 1}^{\pm}, x_{j, 0}^{\pm}] - [x_{i, 0}^{\pm}, x_{j, 1}^{\pm}] = \pm a_{ij}\dfrac{\varepsilon_1 + \varepsilon_2}{2} \{x_{i, 0}^{\pm}, x_{j, 0}^{\pm}\} - m_{ij} \dfrac{\varepsilon_1 - \varepsilon_2}{2} [x_{i, 0}^{\pm}, x_{j, 0}^{\pm}],\label{eq2.6}\\
(\ad x_{i,0}^{\pm})^{1+|a_{ij}|} (x_{j,0}^{\pm})= 0\ (i \neq j), \label{eq2.7}\\
[x^\pm_{i,0},x^\pm_{i,0}]=0\ (i=0, m),\label{eq2.8}\\
[[x^\pm_{i-1,0},x^\pm_{i,0}],[x^\pm_{i,0},x^\pm_{i+1,0}]]=0\ (i=0, m),\label{eq2.9}
\end{gather}
where the generators $x^\pm_{m, r}$ and $x^\pm_{0, r}$ are odd and all other generators are even.
\end{Theorem}
Since the definition of the affine super Yangian is very complicated, it is not clear whether the affine super Yangian is trivial or not. However, there exists the non-trivial homomorphism from the affine super Yangian to the completion of $U(\widehat{\mathfrak{gl}}(m|n))$. This homomorphism is called as the evaluation map. In order to introduce the evaluation map, we set some notations.

First, let us recall the definition of a Lie superalgebra $\widehat{\mathfrak{gl}}(m|n)$. We set a Lie superalgebra $\widehat{\mathfrak{gl}}(m|n)$ as $\mathfrak{gl}(m|n)\otimes\mathbb{C}[t^{\pm1}]\oplus\mathbb{C}c\oplus\mathbb{C}z$ whose commutator relations are following;
\begin{gather*}
[x\otimes t^u, y\otimes t^v]=\begin{cases}
[x,y]\otimes t^{u+v}+u\delta_{u+v,0}\text{str}(xy)c\quad\text{ if }x,\ y\in\mathfrak{sl}(m|n),\\
[e_{a,b},e_{i,i}]\otimes t^{u+v}+u\delta_{u+v,0}\text{str}(e_{a,b}e_{i,i})c+u\delta_{u+v,0}\delta_{a,b}{(-1)}^{p(a)+p(i)}z\\
\qquad\qquad\qquad\qquad\qquad\qquad\qquad\text{ if }x=e_{a,b},\ y=e_{i,i},
\end{cases}\\
\text{$c$ and $z$ are central elements of }\widehat{\mathfrak{gl}}(m|n),
\end{gather*}
where $\text{str}$ is a supertrace of $\mathfrak{gl}(m|n)$.
In this section, we assume that the central element $c$ and $z$ are not indeterminates but complex numbers. Taking the degree of $U(\widehat{\mathfrak{gl}}(m|n))$ determined by
\begin{equation*}
\text{deg}(E_{i,j}(s))=s,\quad\text{deg}(c)=\text{deg}(z)=0,
\end{equation*}
we define $U(\widehat{\mathfrak{gl}}(m|n))_{{\rm comp}}$ as the degreewise completion of $U(\widehat{\mathfrak{gl}}(m|n))$ in the sense of \cite{MNT}. 
Let us set
\begin{gather*}
 \delta(i\leq j)=\begin{cases}
 1&(i\leq j),\\
 0&(i>j).
 \end{cases},\\
 h_i=\begin{cases}
-E_{1,1}-E_{m+n,m+n}+c&(i=0),\\
{(-1)}^{p(i)}E_{ii}-{(-1)}^{p(i+1)}E_{i+1,i+1}&(1\leq i\leq m+n-1),
\end{cases}\\
x^+_i=\begin{cases}
E_{m+n,1}\otimes t&(i=0),\\
E_{i,i+1}&(\text{otherwise}),
\end{cases}
\quad x^-_i=\begin{cases}
-E_{1,m+n}\otimes t^{-1}&(i=0),\\
{(-1)}^{p(i)}E_{i+1,i}&(\text{otherwise}).
\end{cases}
\end{gather*}
Then, we are in the position to introduce the evaluation map.
\begin{Theorem}[\cite{U2}, Theorem~5.1]\label{thm:main}
Let $\alpha$ be a complex number. We also assume that $\hbar c =(-m+n) \ve_1$ and $z=1$. 
Then, there exists an algebra homomorphism 
$
\ev_{\ve_1,\ve_2} \colon Y_{\ve_1,\ve_2}(\widehat{\mathfrak{sl}}(m|n)) \to U(\widehat{\mathfrak{gl}}(m|n))_{{\rm comp}}
$
uniquely determined by 
\begin{gather*}
	\ev_{\ve_1,\ve_2}(x_{i,0}^{+}) =x^+_i,\quad\ev_{\ve_1,\ve_2}(x_{i,0}^{-}) = x^-_i,\quad\ev_{\ve_1,\ve_2}(h_{i,0}) =h_i,
\end{gather*}
\begin{gather*}
	\ev_{\ve_1,\ve_2}(h_{i,1}) = \begin{cases}
		(\alpha - (m-n) \varepsilon_1)h_{0} + \hbar E_{m+n,m+n} (E_{1,1}-c) \\
		\ -\hbar \displaystyle\sum_{s \geq 0} \limits\displaystyle\sum_{k=1}^{m+n}\limits{(-1)}^{p(k)}E_{m+n,k}(-s) E_{k,m+n}(s)\\
		\quad-\hbar \displaystyle\sum_{s \geq 0}\limits\displaystyle\sum_{k=1}^{m+n}\limits{(-1)}^{p(k)}E_{1,k}(-s-1) E_{k,1}(s+1)\qquad\qquad\qquad\qquad\text{ if $i = 0$},\\
		\\
		(\alpha - (i-2\delta(i\geq m+1)(i-m))\varepsilon_1) h_{i} -{(-1)}^{p(E_{i,i+1})} \hbar E_{i,i}E_{i+1,i+1} \\
		\ + \hbar{(-1)}^{p(i)} \displaystyle\sum_{s \geq 0}  \limits\displaystyle\sum_{k=1}^{i}\limits{(-1)}^{p(k)} E_{i,k}(-s) E_{k,i}(s)\\
		\quad +\hbar{(-1)}^{p(i)} \displaystyle\sum_{s \geq 0} \limits\displaystyle\sum_{k=i+1}^{m+n}\limits {(-1)}^{p(k)}E_{i,k}(-s-1) E_{k,i}(s+1) \\
		\qquad -\hbar{(-1)}^{p(i+1)}\displaystyle\sum_{s \geq 0}\limits\displaystyle\sum_{k=1}^{i}\limits{(-1)}^{p(k)}E_{i+1,k}(-s) E_{k,i+1}(s)\\
		 \qquad\quad-\hbar{(-1)}^{p(i+1)}\displaystyle\sum_{s \geq 0}\limits\displaystyle\sum_{k=i+1}^{m+n} \limits{(-1)}^{p(k)}E_{i+1,k}(-s-1) E_{k,i+1}(s+1)\\
		\qquad\qquad\qquad\qquad\qquad\qquad\qquad\qquad\qquad\qquad\qquad\qquad\qquad\qquad \text{ if $i \neq 0$},
	\end{cases}
\end{gather*}
where $\hbar=\ve_1+\ve_2$ and $E_{i,j}(s)=E_{i,j}\otimes t^s$.
\end{Theorem}
\begin{Remark}
In the case when $\mathfrak{g}$ is $\widehat{\mathfrak{sl}}(n)$, the evaluation map is defined in \cite{Gu1} and \cite{K1}. In this case, Kodera \cite{K2} showed that the image of the evaluation is dense in the completion of $U(\widehat{\mathfrak{gl}}(m|n)$. However, since the proof in \cite{K2} needs a braid group, we cannot prove the same statement in the super setting.
\end{Remark}
\section{The surjectivity of the evaluation map}
In this section, we show that the image of $\ev_{\ve_1,\ve_2}$ is dense in the completion of $U(\widehat{\mathfrak{gl}}(m|n))$ provided that $\ve_1\neq0$. 
By the definition of $\ev_{\ve_1,\ve_2}$, the image of $\ev_{\ve_1,\ve_2}$ contains $h_i$ and $x^\pm_i$. Since $h_i$ and $x^\pm_i$ are generators of $\widehat{\mathfrak{sl}}(m|n)$, the image of $\ev_{\ve_1,\ve_2}$ contains $\widehat{\mathfrak{sl}}(m|n)$. Thus, it is enough to prove that the image of $\ev_{\ve_1,\ve_2}$ contains $E_{i,i}(s)$ for all $1\leq i\leq m+n$ and $s\in\mathbb{Z}$.

First, we show that the image of $\ev_{\ve_1,\ve_2}$ contains $E_{i,i}(0)$.
\begin{Theorem}\label{T0}
We obtain
\begin{align}
&\quad\ev_{\ve_1,\ve_2}(\displaystyle\sum_{0\leq i \leq m+n-1}\limits \tilde{h}_{i,1})\nonumber\\
&=(\alpha - (m-n) \varepsilon_1)h_{0}+\displaystyle\sum_{1\leq i\leq m+n-1}\limits(\alpha - (i-2\delta(i\geq m+1)(i-m))\varepsilon_1) h_{i}-c\hbar E_{m+n,m+n}.\label{942}
\end{align}
\end{Theorem}
\begin{proof}
By the definition of $\ev_{\ve_1,\ve_2}(h_{i,1})$, we obtain
\begin{gather*}
	\ev_{\ve_1,\ve_2}(\tilde{h}_{i,1}) = \begin{cases}
		(\alpha - (m-n) \varepsilon_1)h_{0} -\dfrac{1}{2} \hbar E_{m+n,m+n}^2-\dfrac{1}{2} \hbar E_{1,1}^2-c\hbar E_{m+n,m+n}\\
		\ -\hbar \displaystyle\sum_{s \geq 0} \limits\displaystyle\sum_{k=1}^{m+n}\limits{(-1)}^{p(k)}E_{m+n,k}(-s) E_{k,m+n}(s)\\
		\quad-\hbar \displaystyle\sum_{s \geq 0}\limits\displaystyle\sum_{k=1}^{m+n}\limits{(-1)}^{p(k)}E_{1,k}(-s-1) E_{k,1}(s+1)\\
		 \qquad\qquad\qquad\qquad\qquad\qquad\qquad\qquad\qquad\qquad\qquad\qquad\qquad\qquad\text{ if $i = 0$},\\
\\
		(\alpha - (i-2\delta(i\geq m+1)(i-m))\varepsilon_1) h_{i} -\dfrac{1}{2} \hbar E_{i,i}^2-\dfrac{1}{2} \hbar E_{i+1,i+1}^2\\
		\ + \hbar{(-1)}^{p(i)} \displaystyle\sum_{s \geq 0} \limits\displaystyle\sum_{k=1}^{i}\limits{(-1)}^{p(k)} E_{i,k}(-s) E_{k,i}(s)\\
		\quad +\hbar{(-1)}^{p(i)} \displaystyle\sum_{s \geq 0} \limits\displaystyle\sum_{k=i+1}^{m+n}\limits {(-1)}^{p(k)}E_{i,k}(-s-1) E_{k,i}(s+1) \\
		\qquad -\hbar{(-1)}^{p(i+1)}\displaystyle\sum_{s \geq 0}\limits\displaystyle\sum_{k=1}^{i}\limits{(-1)}^{p(k)}E_{i+1,k}(-s) E_{k,i+1}(s)\\
		 \qquad\quad-\hbar{(-1)}^{p(i+1)}\displaystyle\sum_{s \geq 0}\limits\displaystyle\sum_{k=i+1}^{m+n} \limits{(-1)}^{p(k)}E_{i+1,k}(-s-1) E_{k,i+1}(s+1)\\
		\qquad\qquad\qquad\qquad\qquad\qquad\qquad\qquad\qquad\qquad\qquad\qquad\qquad\qquad \text{ if $i \neq 0$}.
	\end{cases}
\end{gather*}
Then, we rewrite the left hand side of \eqref{942} as
\begin{align}
&(\alpha - (m-n) \varepsilon_1)h_{0}+\displaystyle\sum_{1\leq i\leq m+n-1}\limits(\alpha - (i-2\delta(i\geq m+1)(i-m))\varepsilon_1) h_{i}\nonumber\\
&\quad-\frac{\hbar}{2}(E_{1,1}^2+E_{m+n,m+n}^2)-c\hbar E_{m+n,m+n}-\displaystyle\sum_{1\leq i\leq m+n-1}\limits\frac{\hbar}{2}(E_{i,i}^2+E_{i+1,i+1}^2)\nonumber\\
&\quad+\displaystyle\sum_{1\leq i\leq m+n}\hbar{(-1)}^{p(i)} \displaystyle\sum_{s \geq 0}  \limits\displaystyle\sum_{k=1}^{i}\limits{(-1)}^{p(k)} E_{i,k}(-s) E_{k,i}(s)\nonumber\\
&\quad+\displaystyle\sum_{1\leq i\leq m+n}\hbar{(-1)}^{p(i)} \displaystyle\sum_{s \geq 0} \limits\displaystyle\sum_{k=i+1}^{m+n}\limits {(-1)}^{p(k)}E_{i,k}(-s-1) E_{k,i}(s+1)\nonumber\\
&\quad-\displaystyle\sum_{0\leq i\leq m+n-1}\hbar{(-1)}^{p(i+1)} \displaystyle\sum_{s \geq 0}  \limits\displaystyle\sum_{k=1}^{i}\limits{(-1)}^{p(k)} E_{i+1,k}(-s) E_{k,i+1}(s)\nonumber\\
&\quad-\displaystyle\sum_{0\leq i\leq m+n-1}\hbar{(-1)}^{p(i+1)} \displaystyle\sum_{s \geq 0} \limits\displaystyle\sum_{k=i+1}^{m+n}\limits {(-1)}^{p(k)}E_{i+1,k}(-s-1) E_{k,i+1}(s+1).\label{943}
\end{align}
Adding the first and third terms of \eqref{943}, we have
\begin{align}
&\displaystyle\sum_{1\leq i\leq m+n}\hbar{(-1)}^{p(i)} \displaystyle\sum_{s \geq 0}  \limits\displaystyle\sum_{k=1}^{i}\limits{(-1)}^{p(k)} E_{i,k}(-s) E_{k,i}(s)\nonumber\\
&\qquad\qquad-\displaystyle\sum_{1\leq j\leq m+n}\hbar{(-1)}^{p(j)} \displaystyle\sum_{s \geq 0}  \limits\displaystyle\sum_{k=1}^{j-1}\limits{(-1)}^{p(k)} E_{j,k}(-s) E_{k,j}(s)\nonumber\\
&=\hbar\displaystyle\sum_{1\leq i\leq m+n}\limits\displaystyle\sum_{s \geq 0} \limits E_{i,i}(-s) E_{i,i}(s).\label{946}
\end{align}
Adding the second and 4-th terms of \eqref{943}, we obtain
\begin{align}
&\displaystyle\sum_{1\leq i\leq m+n}\hbar{(-1)}^{p(i)} \displaystyle\sum_{s \geq 0} \limits\displaystyle\sum_{k=i+1}^{m+n}\limits {(-1)}^{p(k)}E_{i,k}(-s-1) E_{k,i}(s+1)\nonumber\\
&\qquad\qquad-\displaystyle\sum_{1\leq j\leq m+n}\hbar{(-1)}^{p(j)} \displaystyle\sum_{s \geq 0} \limits\displaystyle\sum_{k=j}^{m+n}\limits {(-1)}^{p(k)}E_{j,k}(-s-1) E_{k,j}(s+1)\nonumber\\
&=-\hbar\displaystyle\sum_{1\leq i\leq m+n}\limits\displaystyle\sum_{s \geq 0} \limits E_{i,i}(-s-1) E_{i,i}(s+1).\label{947}
\end{align}
Applying \eqref{946} and \eqref{947} to \eqref{943}, we find that the left hand side of \eqref{942} is equal to
\begin{align*}
&(\alpha - (m-n) \varepsilon_1)h_{0}+\displaystyle\sum_{1\leq i\leq m+n-1}\limits(\alpha - (i-2\delta(i\geq m+1)(i-m))\varepsilon_1) h_{i}\nonumber\\
&\quad-\frac{\hbar}{2}(E_{1,1}^2+E_{m+n,m+n}^2)-c\hbar E_{m+n,m+n}-\displaystyle\sum_{1\leq i\leq m+n-1}\limits\frac{\hbar}{2}(E_{i,i}^2+E_{i+1,i+1}^2)+\hbar\displaystyle\sum_{1\leq i\leq m+n}\limits E_{i,i}^2.
\end{align*}
By direct computation, it is equal to
\begin{align*}
(\alpha - (m-n) \varepsilon_1)h_{0}+\displaystyle\sum_{1\leq i\leq m+n-1}\limits(\alpha - (i-2\delta(i\geq m+1)(i-m))\varepsilon_1) h_{i}-c\hbar E_{m+n,m+n}.
\end{align*}
Thus, we have obtained Theorem~\ref{T0}.
\end{proof}
Since $h_i$ is contained in the image of $\ev_{\ve_1,\ve_2}$, the image of $\ev_{\ve_1,\ve_2}$ contains $c\hbar E_{m+n,m+n}$.
\begin{Corollary}\label{c1}
The image of $\ev_{\ve_1,\ve_2}$ contains $E_{m+n,m+n}$ provided that $\hbar c\neq0$.
\end{Corollary}
Next, let us show that the completion of the image of $\ev_{\ve_1,\ve_2}$ contains $E_{i,i}(s)\ (s\neq0)$.
\begin{Theorem}\label{Thm125}
For all $i\neq0$, we obtain
\begin{align*}
&\quad[\ev_{\ve_1,\ve_2}(h_{i,1}),({(-1)}^{p(i)}E_{i,i}-{(-1)}^{p(i+1)}E_{i+1,i+1})t^a]\\
&=\hbar\displaystyle\sum_{s \geq 0}  \limits \delta_{s+a,0}scE_{i,i}(-s)-\hbar\displaystyle\sum_{s \geq 0}  \limits \delta_{-s+a,0}scE_{i,i}(s)\\
&\quad+\hbar\displaystyle\sum_{s \geq 0}\limits \delta_{s+1+a,0}(s+1)cE_{i+1,i+1}(-s-1)-\hbar\displaystyle\sum_{s \geq 0}\limits \delta_{-s-1+a,0}(s+1)cE_{i+1,i+1}(s+1)\\
&\quad+\text{sum of elements of the completion of $U(\widehat{\mathfrak{sl}}(m|n))$}.
\end{align*}
\end{Theorem}
\begin{proof}
The proof is done by direct computation.
By the definition of $\ev_{\ve_1,\ve_2}(h_{i,1})$, we have
\begin{align}
&\quad[\ev_{\ve_1,\ve_2}(h_{i,1}),({(-1)}^{p(i)}E_{i,i}-{(-1)}^{p(i+1)}E_{i+1,i+1})t^a]\nonumber\\
&=[(\alpha - (i-2\delta(i\geq m+1)(i-m))\varepsilon_1) h_{i} ,({(-1)}^{p(i)}E_{i,i}-{(-1)}^{p(i+1)}E_{i+1,i+1})t^a]\nonumber\\
&\quad-{(-1)}^{p(E_{i,i+1})} \hbar [E_{i,i}E_{i+1,i+1}, ({(-1)}^{p(i)}E_{i,i}-{(-1)}^{p(i+1)}E_{i+1,i+1})t^a]\nonumber\\
&\quad+[\hbar{(-1)}^{p(i)} \displaystyle\sum_{s \geq 0}  \limits\displaystyle\sum_{k=1}^{i}\limits{(-1)}^{p(k)} E_{i,k}(-s) E_{k,i}(s),({(-1)}^{p(i)}E_{i,i}-{(-1)}^{p(i+1)}E_{i+1,i+1})t^a]\nonumber\\
&\quad+[\hbar{(-1)}^{p(i)} \displaystyle\sum_{s \geq 0} \limits\displaystyle\sum_{k=i+1}^{m+n}\limits {(-1)}^{p(k)}E_{i,k}(-s-1) E_{k,i}(s+1),({(-1)}^{p(i)}E_{i,i}-{(-1)}^{p(i+1)}E_{i+1,i+1})t^a]\nonumber\\
&\quad -[\hbar{(-1)}^{p(i+1)}\displaystyle\sum_{s \geq 0}\limits\displaystyle\sum_{k=1}^{i}\limits{(-1)}^{p(k)}E_{i+1,k}(-s) E_{k,i+1}(s),({(-1)}^{p(i)}E_{i,i}-{(-1)}^{p(i+1)}E_{i+1,i+1})t^a]\nonumber\\
&\quad-[\hbar{(-1)}^{p(i+1)}\displaystyle\sum_{s \geq 0}\limits\displaystyle\sum_{k=i+1}^{m+n} \limits{(-1)}^{p(k)}E_{i+1,k}(-s-1) E_{k,i+1}(s+1),\nonumber\\
&\qquad\qquad\qquad\qquad\qquad\qquad\qquad\qquad\qquad\qquad\qquad({(-1)}^{p(i)}E_{i,i}-{(-1)}^{p(i+1)}E_{i+1,i+1})t^a].\label{190}
\end{align}
We can rewrite each terms of the right hand side of \eqref{190}. By an easy computation, we find that the first two terms of the right hand side of \eqref{190}. Other terms are computed as follows.
\begin{Claim}\label{Claim1}
\textup{(1)}\ The 4-th and 5-th terms of the right hand side of \eqref{190} are elements of the completion of $U(\widehat{\mathfrak{sl}}(m|n))$.

\textup{(2)}\ We can rewrite the third term of the right hand side of \eqref{190} as
\begin{align} 
&\hbar{(-1)}^{p(i)} \displaystyle\sum_{s \geq 0}  \limits \delta_{s+a,0}scE_{i,i}(-s)-\hbar\displaystyle\sum_{s \geq 0}  \limits \delta_{-s+a,0}scE_{i,i}(s)\nonumber\\
&\quad+\text{an element of the completion of $U(\widehat{\mathfrak{sl}}(m|n))$}.\label{920}
\end{align}
\textup{(3)}\ We can rewrite 6-th term of the right hand side of \eqref{190} as
\begin{align}
&\hbar\displaystyle\sum_{s \geq 0}\limits \delta_{s+1+a,0}(s+1)cE_{i+1,i+1}(-s-1)-\hbar\displaystyle\sum_{s \geq 0}\limits \delta_{-s-1+a,0}(s+1)cE_{i+1,i+1}(s+1)\nonumber\\
&\quad+\text{an element of the completion of $U(\widehat{\mathfrak{sl}}(m|n))$}.\label{921}
\end{align}
\end{Claim}
Assuming Claim~\ref{Claim1}, we obtain Theorem~\ref{Thm125} by adding \eqref{920} and \eqref{921}.
In order to complete the proof of Theorem~\ref{Thm125}, we prove Claim~\ref{Claim1}.
\begin{proof}[the proof of Claim~\ref{Claim1}]
\textup{(1)}\ The proof is due to direct computation.
First we prove the 4-th case. We can rewrite the 4-th term of the right hand side of \eqref{190} as follows;
\begin{align}
&\hbar{(-1)}^{p(i)} \displaystyle\sum_{s \geq 0} \limits\displaystyle\sum_{k=i+1}^{m+n}\limits {(-1)}^{p(k)}E_{i,k}(-s-1) [E_{k,i}(s+1),({(-1)}^{p(i)}E_{i,i}-{(-1)}^{p(i+1)}E_{i+1,i+1})t^a]\nonumber\\
&\quad+\hbar{(-1)}^{p(i)} \displaystyle\sum_{s \geq 0} \limits\displaystyle\sum_{k=i+1}^{m+n}\limits {(-1)}^{p(k)}[E_{i,k}(-s-1),({(-1)}^{p(i)}E_{i,i}-{(-1)}^{p(i+1)}E_{i+1,i+1})t^a] E_{k,i}(s+1).\label{200}
\end{align}
We rewrite each terms of the right hand side of \eqref{200}. By direct computation, we can rewrite the first term of the right hand side of \eqref{200} as
\begin{align}
&\hbar\displaystyle\sum_{s \geq 0} \limits\displaystyle\sum_{k=i+1}^{m+n}\limits {(-1)}^{p(k)}E_{i,k}(-s-1)E_{k,i}(s+1+a)\nonumber\\
&\qquad\qquad+\hbar{(-1)}^{p(i)+p(i+1)} \displaystyle\sum_{s \geq 0} \limits{(-1)}^{p(i+1)}E_{i,i+1}(-s-1)E_{i+1,i}(s+1+a).\label{201}
\end{align}
By direct computation, we can rewrite the second term of the right hand side of \eqref{200} as
\begin{align}
&-\hbar{(-1)}^{p(i)} \displaystyle\sum_{s \geq 0} \limits E_{i,i+1}(-s-1+a)E_{i+1,i}(s+1)\nonumber\\
&\qquad\qquad-\hbar\displaystyle\sum_{s \geq 0} \limits\displaystyle\sum_{k=i+1}^{m+n}\limits {(-1)}^{p(k)}E_{i,k}(-s-1+a) E_{k,i}(s+1).\label{202}
\end{align}
Adding \eqref{201} and \eqref{202}, we obtain
\begin{align}
&\quad\text{the 4-th term of the right hand side of \eqref{190}}\nonumber\\
&=\hbar\displaystyle\sum_{s \geq 0} \limits\displaystyle\sum_{k=i+1}^{m+n}\limits {(-1)}^{p(k)}E_{i,k}(-s-1)E_{k,i}(s+1+a)+\hbar{(-1)}^{p(i)} \displaystyle\sum_{s \geq 0} \limits E_{i,i+1}(-s-1)E_{i+1,i}(s+1+a)\nonumber\\
&\quad-\hbar{(-1)}^{p(i)} \displaystyle\sum_{s \geq 0} \limits E_{i,i+1}(-s-1+a)E_{i+1,i}(s+1)\nonumber\\
&\quad-\hbar\displaystyle\sum_{s \geq 0} \limits\displaystyle\sum_{k=i+1}^{m+n}\limits {(-1)}^{p(k)}E_{i,k}(-s-1+a) E_{k,i}(s+1),\label{204}
\end{align}
Since all of the terms of the right hand side of \eqref{204} are elements of the completion of $U(\widehat{\mathfrak{sl}}(m|n))$, the 4-th term of the right hand side of \eqref{190} is an element of the completion of $U(\widehat{\mathfrak{sl}}(m|n))$.

Next, we prove the 5-th case. Let us rewrite the 5-th term of the right hand side of \eqref{190} as follows;
\begin{align}
&-\hbar{(-1)}^{p(i+1)}\displaystyle\sum_{s \geq 0}\limits\displaystyle\sum_{k=1}^{i}\limits{(-1)}^{p(k)}E_{i+1,k}(-s) [E_{k,i+1}(s),({(-1)}^{p(i)}E_{i,i}-{(-1)}^{p(i+1)}E_{i+1,i+1})t^a]\nonumber\\
&\quad-\hbar{(-1)}^{p(i+1)}\displaystyle\sum_{s \geq 0}\limits\displaystyle\sum_{k=1}^{i}\limits{(-1)}^{p(k)}[E_{i+1,k}(-s),({(-1)}^{p(i)}E_{i,i}-{(-1)}^{p(i+1)}E_{i+1,i+1})t^a] E_{k,i+1}(s).\label{211}
\end{align}
We rewrite each terms of the right hand side of \eqref{211}. By direct computation, we can rewrite the first term of the right hand side of \eqref{211} as
\begin{align}
&\hbar\displaystyle\sum_{s \geq 0}\limits\displaystyle\sum_{k=1}^{i}\limits{(-1)}^{p(k)}E_{i+1,k}(-s)E_{k,i+1}(s+a)+\hbar{(-1)}^{p(i+1)}\displaystyle\sum_{s \geq 0}\limits E_{i+1,i}(-s) E_{i,i+1}(s+a).\label{212}
\end{align}
By direct computation, we can also rewrite the first term of the right hand side of \eqref{211} as
\begin{align}
&-\hbar{(-1)}^{p(i+1)}\displaystyle\sum_{s \geq 0}\limits E_{i+1,i}(-s+a)E_{i,i+1}(s)-\hbar\displaystyle\sum_{s \geq 0}\limits\displaystyle\sum_{k=1}^{i}\limits{(-1)}^{p(k)}E_{i+1,k}(a-s)E_{k,i+1}(s).\label{213}
\end{align}
Adding \eqref{212} and \eqref{213}, we have
\begin{align}
&\quad\text{the 5-th term of the right hand side of \eqref{190}}\nonumber\\
&=\hbar\displaystyle\sum_{s \geq 0}\limits\displaystyle\sum_{k=1}^{i}\limits{(-1)}^{p(k)}E_{i+1,k}(-s)E_{k,i+1}(s+a)+\hbar{(-1)}^{p(i+1)}\displaystyle\sum_{s \geq 0}\limits E_{i+1,i}(-s) E_{i,i+1}(s+a)\nonumber\\
&\quad-\hbar{(-1)}^{p(i+1)}\displaystyle\sum_{s \geq 0}\limits E_{i+1,i}(-s+a)E_{i,i+1}(s)-\hbar\displaystyle\sum_{s \geq 0}\limits\displaystyle\sum_{k=1}^{i}\limits{(-1)}^{p(k)}E_{i+1,k}(a-s)E_{k,i+1}(s).\label{214}
\end{align}
Since all of the terms of the right hand side of \eqref{214} are elements of the completion of $U(\widehat{\mathfrak{sl}}(m|n))$, the 5-th term of the right hand side of \eqref{190} is an element of the completion of $U(\widehat{\mathfrak{sl}}(m|n))$.

\textup{(2)}\ 
The proof is due to direct computation. 
Let us rewrite the third term of the right hand side of \eqref{190} as follows;
\begin{align}
&\hbar{(-1)}^{p(i)} \displaystyle\sum_{s \geq 0}  \limits\displaystyle\sum_{k=1}^{i}\limits{(-1)}^{p(k)} E_{i,k}(-s) [E_{k,i}(s),({(-1)}^{p(i)}E_{i,i}-{(-1)}^{p(i+1)}E_{i+1,i+1})t^a]\nonumber\\
&\quad+\hbar{(-1)}^{p(i)} \displaystyle\sum_{s \geq 0}  \limits\displaystyle\sum_{k=1}^{i}\limits{(-1)}^{p(k)} [E_{i,k}(-s),({(-1)}^{p(i)}E_{i,i}-{(-1)}^{p(i+1)}E_{i+1,i+1})t^a] E_{k,i}(s).\label{191}
\end{align}
We rewrite each terms of the right hand side of \eqref{191}. By direct computation, we can rewrite the first term of the right hand side of \eqref{191} as
\begin{align}
&\hbar\displaystyle\sum_{s \geq 0}  \limits\displaystyle\sum_{k=1}^{i}\limits{(-1)}^{p(k)} E_{i,k}(-s) E_{k,i}(s+a)-\hbar{(-1)}^{p(i)} \displaystyle\sum_{s \geq 0}  \limits E_{i,i}(-s)E_{i,i}(s+a)\nonumber\\
&\quad+\hbar\displaystyle\sum_{s \geq 0}  \limits \delta_{s+a,0}scE_{i,i}(-s)+\hbar\displaystyle\sum_{s \geq 0}  \limits \delta_{s+a,0}sE_{i,i}(-s)-\hbar\displaystyle\sum_{s \geq 0}  \limits \delta_{s+a,0}sE_{i,i}(-s)\nonumber\\
&=\hbar\displaystyle\sum_{s \geq 0}  \limits\displaystyle\sum_{k=1}^{i-1}\limits{(-1)}^{p(k)} E_{i,k}(-s) E_{k,i}(s+a)+\hbar \displaystyle\sum_{s \geq 0}  \limits \delta_{s+a,0}scE_{i,i}(-s).\label{192}
\end{align}
Similarly, we can rewrite the second term of the right hand side of \eqref{191} as
\begin{align}
&\hbar{(-1)}^{p(i)} \displaystyle\sum_{s \geq 0}  \limits E_{i,i}(-s+a)E_{i,i}(s)-\hbar\displaystyle\sum_{s \geq 0}  \limits\displaystyle\sum_{k=1}^{i}\limits{(-1)}^{p(k)} E_{i,k}(-s+a) E_{k,i}(s)\nonumber\\
&\quad-\hbar\displaystyle\sum_{s \geq 0}  \limits \delta_{-s+a,0}scE_{i,i}(s)-\hbar\displaystyle\sum_{s \geq 0}  \limits \delta_{-s+a,0}sE_{i,i}(s)+\hbar\displaystyle\sum_{s \geq 0}  \limits \delta_{-s+a,0}sE_{i,i}(s)\nonumber\\
&=-\hbar\displaystyle\sum_{s \geq 0}  \limits\displaystyle\sum_{k=1}^{i-1}\limits{(-1)}^{p(k)} E_{i,k}(-s+a) E_{k,i}(s)-\hbar\displaystyle\sum_{s \geq 0}  \limits \delta_{-s+a,0}scE_{i,i}(s).\label{193}
\end{align}
Adding \eqref{192} and \eqref{193}, we obtain
\begin{align}
&\quad\text{the third term of the right hand side of \eqref{190}}\nonumber\\
&=\hbar\displaystyle\sum_{s \geq 0}  \limits\displaystyle\sum_{k=1}^{i-1}\limits{(-1)}^{p(k)} E_{i,k}(-s) E_{k,i}(s+a)-\hbar\displaystyle\sum_{s \geq 0}  \limits\displaystyle\sum_{k=1}^{i-1}\limits{(-1)}^{p(k)} E_{i,k}(-s+a) E_{k,i}(s)\nonumber\\
&\quad+\hbar{(-1)}^{p(i)} \displaystyle\sum_{s \geq 0}  \limits \delta_{s+a,0}scE_{i,i}(-s)-\hbar\displaystyle\sum_{s \geq 0}  \limits \delta_{-s+a,0}scE_{i,i}(s).\label{551}
\end{align}
Since the first two terms of the right hand side of \eqref{551} are elements of the completion of $U(\widehat{\mathfrak{sl}}(m|n))$, we have obtained \eqref{920}.

\textup{(3)}\ We rewrite the 6-th term of the right hand side of \eqref{190} as follows;
\begin{align}
&-\hbar{(-1)}^{p(i+1)}\displaystyle\sum_{s \geq 0}\limits\displaystyle\sum_{k=i+1}^{m+n} \limits{(-1)}^{p(k)}E_{i+1,k}(-s-1) \nonumber\\
&\qquad\qquad\qquad\qquad\qquad\qquad\qquad\qquad[E_{k,i+1}(s+1),({(-1)}^{p(i)}E_{i,i}-{(-1)}^{p(i+1)}E_{i+1,i+1})t^a]\nonumber\\
&\quad-\hbar{(-1)}^{p(i+1)}\displaystyle\sum_{s \geq 0}\limits\displaystyle\sum_{k=i+1}^{m+n} \limits{(-1)}^{p(k)}[E_{i+1,k}(-s-1),\nonumber\\
&\qquad\qquad\qquad\qquad\qquad\qquad\qquad\qquad({(-1)}^{p(i)}E_{i,i}-{(-1)}^{p(i+1)}E_{i+1,i+1})t^a]E_{k,i+1}(s+1).\label{221}
\end{align}
We compute each terms of the right hand side of \eqref{221}. By direct computation, we can rewrite the first term of the right hand side of \eqref{221} as
\begin{align}
&\hbar\displaystyle\sum_{s \geq 0}\limits\displaystyle\sum_{k=i+1}^{m+n} \limits{(-1)}^{p(k)}E_{i+1,k}(-s-1)E_{k,i+1}(s+1+a)\nonumber\\
&\quad-\hbar{(-1)}^{p(i+1)}\displaystyle\sum_{s \geq 0}\limits E_{i+1,i+1}(-s-1)E_{i+1,i+1}(s+1+a)+\hbar\displaystyle\sum_{s \geq 0}\limits (s+1)c\delta_{s+1+a,0}E_{i+1,i+1}(-s-1)\nonumber\\
&\quad-\hbar\displaystyle\sum_{s \geq 0}\limits (s+1)\delta_{s+1+a,0}E_{i+1,i+1}(-s-1)+\hbar\displaystyle\sum_{s \geq 0}\limits (s+1)\delta_{s+1+a,0}E_{i+1,i+1}(-s-1)\nonumber\\
&=\hbar\displaystyle\sum_{s \geq 0}\limits\displaystyle\sum_{k=i+2}^{m+n} \limits{(-1)}^{p(k)}E_{i+1,k}(-s-1)E_{k,i+1}(s+1+a)+\hbar\displaystyle\sum_{s \geq 0}\limits \delta_{s+1+a,0}(s+1)cE_{i+1,i+1}(-s-1).\label{222}
\end{align}
By direct computation, we can also rewrite the second term of the right hand side of \eqref{221} as
\begin{align}
&\hbar{(-1)}^{p(i+1)}\displaystyle\sum_{s \geq 0}\limits E_{i+1,i+1}(-s-1+a)E_{i+1,i+1}(s+1)\nonumber\\
&\quad-\hbar\displaystyle\sum_{s \geq 0}\limits\displaystyle\sum_{k=i+1}^{m+n} \limits{(-1)}^{p(k)}E_{i+1,k}(-s-1+a)E_{k,i+1}(s+1)-\hbar\displaystyle\sum_{s \geq 0}\limits \delta_{-s-1+a,0}(s+1)cE_{i+1,i+1}(s+1)\nonumber\\
&\quad+\hbar\displaystyle\sum_{s \geq 0}\limits \delta_{-s-1+a,0}(s+1)E_{i+1,i+1}(s+1)-\hbar\displaystyle\sum_{s \geq 0}\limits \delta_{-s-1+a,0}(s+1)E_{i+1,i+1}(s+1)\nonumber\\
&=-\hbar\displaystyle\sum_{s \geq 0}\limits\displaystyle\sum_{k=i+2}^{m+n} \limits{(-1)}^{p(k)}E_{i+1,k}(-s-1+a)E_{k,i+1}(s+1)\nonumber\\
&\quad-\hbar\displaystyle\sum_{s \geq 0}\limits \delta_{-s-1+a,0}(s+1)cE_{i+1,i+1}(s+1).\label{223}
\end{align}
Adding \eqref{222} and \eqref{223}, we have
\begin{align}
&\quad\text{the 6-th term of the right hand side of \eqref{190}}\nonumber\\
&=\hbar\displaystyle\sum_{s \geq 0}\limits\displaystyle\sum_{k=i+2}^{m+n} \limits{(-1)}^{p(k)}E_{i+1,k}(-s-1)E_{k,i+1}(s+1+a)\nonumber\\
&\quad-\hbar\displaystyle\sum_{s \geq 0}\limits\displaystyle\sum_{k=i+2}^{m+n} \limits{(-1)}^{p(k)}E_{i+1,k}(-s-1+a)E_{k,i+1}(s+1)\nonumber\\
&\quad+\hbar\displaystyle\sum_{s \geq 0}\limits \delta_{s+1+a,0}(s+1)cE_{i+1,i+1}(-s-1)-\hbar\displaystyle\sum_{s \geq 0}\limits \delta_{-s-1+a,0}(s+1)cE_{i+1,i+1}(s+1).\label{215}
\end{align}
Since the first two terms of the right hand side of \eqref{215} are elements of the completion of $U(\widehat{\mathfrak{sl}}(m|n))$, we have obtained \eqref{921}.
\end{proof}
This completes the proof of Theorem~\ref{Thm125}.
\end{proof}
By the assumption that $m,n\geq2$ and $m\neq n$, we can take $1\leq i\leq m+n-1$ such that $p(i)=p(i+1)$. By Theorem~\ref{Thm125}, The completion of the image of $\ev_{\ve_1,\ve_2}$ contains $\hbar c(E_{i,i}+E_{i+1,i+1})t^a$ for all $a\neq0$. Provided that $\hbar c\neq0$, the completion of the image of $\ev_{\ve_1,\ve_2}$ contains $(E_{i,i}+E_{i+1,i+1})t^a$. By the assumption that  $p(i)=p(i+1)$, $(E_{i,i}+E_{i+1,i+1})t^a$ is not contained in $\widehat{\mathfrak{sl}}(m|n)$. Thus, we obtain the following corollary.
\begin{Corollary}\label{c2}
The completion of the image of $\ev_{\ve_1,\ve_2}$ contains $E_{i,i}t^a$ for all $a\neq0$ provided that $\hbar c\neq0$.
\end{Corollary}
By the assumption that $\hbar c=-(m-n)\ve_1$, we find that $\hbar c$ is nonzero if and only if $\ve_1\neq0$. Under the assumption that  by Corollary~\ref{c1} and Corollary~\ref{c2}, the image of  $\ev_{\ve_1,\ve_2}$ contains $E_{i,i}t^s$ for all $s\in\mathbb{Z}$. Thus, we have the following theorem.
\begin{Theorem}
Provided that $\ve_1\neq0$, the image of $\ev_{\ve_1,\ve_2}$ is dense in $U(\widehat{\mathfrak{gl}}(m|n))_{{\rm comp}}$.
\end{Theorem}
\bibliographystyle{plain}
\bibliography{syuu}
\vspace{15 mm}
affiliation: Graduate School of Science,
Kyoto University, Kyoto 606-8502, Japan\newline
email: udmamoru@kurims.kyoto-u.ac.jp

\end{document}